\documentclass{jloganal}
\pdfoutput=1
\usepackage{pinlabel}
\usepackage{amsmath}
\usepackage{amssymb}
\usepackage{amsthm}
\usepackage{amscd}
\usepackage{graphicx}
\usepackage{color}%permite inclusion de color
\usepackage{graphpap}%permite hacer grillas en el entorno picture
\usepackage[mathscr]{euscript}

\title[Sheaves of metric structures]{Sheaves of metric structures}

\author[M. Ochoa]{Maicol A. Ochoa}
\givenname{Maicol}
\surname{Ochoa}
\address{Departamento de Matem\'aticas\\Universidad Nacional de Colombia\\
Carrera 45 N 26-85 - Edificio Uriel Guti\'errez\\
Bogot\'a D.C. - Colombia}
\email{maochoad@unal.edu.co}
\urladdr{}

\author[A. Villaveces]{Andr\'es Villaveces}
\givenname{Andr\'es}
\surname{Villaveces}
\address{Departamento de Matem\'aticas\\Universidad Nacional de Colombia\\
Carrera 45 N 26-85 - Edificio Uriel Guti\'errez\\
Bogot\'a D.C. - Colombia}
\email{avillavecesn@unal.edu.co}
\urladdr{}


\arxivreference{arXiv:1110.4919v1}
\arxivpassword{zuxbs}

\volumenumber{}
\issuenumber{}
\publicationyear{}
\papernumber{}
\startpage{}
\endpage{}
\doi{}
\MR{}
\Zbl{}
\received{}
\revised{}
\accepted{}
\published{}
\publishedonline{}
\proposed{}
\seconded{}
\corresponding{}
\editor{}
\version{}

\newtheorem{theorem}{Theorem}[section]    % Standard theorem environment
\newtheorem{lemma}[theorem]{Lemma}          % Lemma environment with numbering 
\newtheorem{proposition}[theorem]{Proposition}    % Proposition environment with numbering 
\theoremstyle{definition}
\newtheorem{definition}[theorem]{Definition}    % Definition environment with 
\newcommand{\A}{\mathfrak{A}}

\newcommand{\X}{\mathcal{X}}

\newcommand{\F}{\mathbb{F}}

\newcommand{\ket}{\rangle}
\newcommand{\bra}{\langle}

\newcommand{\ts}{\dot{-}}
\newcommand{\eqclass}[1]{[\sigma_{#1}]/_{\sim \F}}
\newcommand{\interior}{\text{\rm int}}
\newcommand{\dom}{\text{\rm dom}}
\begin{document}

\begin{abstract}
 We introduce and develop the Model Theory of Metric Sheaves. A metric sheaf
 $\A$ is defined on a topological space $X$ such that each fiber is a
 metric model. We describe the construction of the generic model as
 the quotient space of the sheaf through an appropriate filter.
 Semantics in this model is completely controlled and understood by
 the forcing rules in the sheaf.
\end{abstract}

\maketitle

\section{Introduction}\label{secintroduction}

 Ultraproducts and ultrapowers play an important role in model theory,
 ring theory, non-standard analysis and other fields in the construction of
 objects with new properties (see for example Pestov~\cite{PEST},
 Robinson \cite{ROBI} and Schoutens \cite{SCHO}). The basic theorems of
 ultraproducts are a specific case of the more general phenomenon of
 model theoretic forcing.
 On the other hand, the logic of sheaves of structures harks back, in a rather implicit
 way, to the work of Grothendieck~\cite{GROT} in his ``Kansas Paper''
 of 1958. Explicit developments of the internal logic on topoi, and of
 the connections between model theoretic forcing and various logics on
 sheaves were successively developed during the decade of 1970 by
 Carson~\cite{CARS}, Ellerman~\cite{ELLE}, Macintyre~\cite{MACI} and
 others. Specifically, Ellerman extracted a version of a
 ``ultrastalk'' theorem linking truth on a generic fiber with model
 theoretic forcing over open sets and Macintyre studied various model
 completions of theories of rings by means of specific constructions
 of sheaves of rings.

 Later, Caicedo\cite{CAIC} introduced in 1995 a more general way to
 construct models by generalizing the previous constructions to sheaves
 of first order structures over arbitrary topological spaces (with no
 Boolean constraints of any kind). In his work he both generalized the
 previous constructions (and versions of a Generic Model Theorem) due
 to Ellerman and Macintyre, and he distilled a substratum of the logic
 on topoi that has several direct applications to classical Model
 Theory.  While other
 constructions  made extensive use of the theory of topoi, Caicedo's
 presentation is substantially simpler, as it does not make explicit
 use of elements of category theory. Every model-theoretical concept
 in these sheaves can be understood by means of the topological
 properties of the sheaf and an appropriate forcing relation. Both the
 fibers and the structures of sections are endowed in a natural way
 with classical structures, but Caicedo's construction provides the
 most natural way of dealing with model theoretic forcing, and ends up
 linking forcing over points and open sets of the topological space
 with classical truth on the ``generic'' structure.

 Sheaves
 of first order structures have been further investigated by Caicedo
 \cite{CAIC2}, Forero \cite{FORE}, Montoya \cite{MONT} and the second
 author of this article (he has applied this approach to the study of
 sheaf-models of Set Theory and generalizing classical forcing over
 partially ordered sets to forcing over arbitrary topological spaces
 in constructing generic models - see~\cite{VILL}).
 
In this
 paper we present a construction that takes these ideas to the realm
 of continuous logic. In brief, we construct sheaves of metric
 structures as understood and studied in the model theory developed by
 Ben Yaacov, Berenstein, Henson and Usvyatsov \cite{HENS}. In the
 following, we assume that the reader is familiar with the results and
 concepts of metric model theory.

  Logical connectives in metric structures are continuous functions
  from $[0,1]^n$ to $[0,1]$ and the supremum and infimum play the role
  of quantifiers.  Semantics differs from that in classical structures
  by the fact that the satisfaction relation is defined on
  $\mathcal{L}$-conditions rather than on $\mathcal{L}$-formulas,
  where $\mathcal{L}$ is a metric signature. If $\phi(x)$ and
  $\psi(y)$ are $\mathcal{L}$-formulas, expressions of the form
  $\phi(x) \leq \psi(y)$, $\phi(x) < \psi(y)$, $\phi(x) \geq \psi(y)$,
  $\phi(x) > \psi(y)$ are $\mathcal{L}$-conditions. In addition, if
  $\phi$ and $\psi$ are sentences then we say that the condition is
  closed. 

 One may argue that the set of connectives is too big. However, the
 set $\mathcal{F}=\{0,1,x/2, \ts\}$, where $0$ and $1$ are taken as
 constant functions, $x/2$ is the function taking half of its input
 and $\ts$ is the truncated subtraction; is uniformly dense in the set
 of all connectives\cite{HENS}. Thus, we can restric the set of
 connectives that we use in building formulas to the set
 $\mathcal{F}$. These constitute the set of $\mathcal{F}$-restricted
 formulas.

 In section \ref{secforcing} we define the sheaf of metric structures,
 introduce pointwise and local forcing on sections and show how to
 define a metric space in some families of sections. In section
 \ref{secgenmodel} we show how to construct the metric generic model
 from a metric sheaf. We also show how the semantics of the generic
 model can be understood by the forcing relation and the topological
 properties of the base space of the sheaf. Finally, we illustrate
 some of our results by means of a simple example. 

\section[Metric sheaf and forcing]{The Metric Sheaf and Forcing}\label{secforcing}

 A  topological sheaf over $X$ is a pair $(E,p)$, where $E$ is a topological space and $p$ is a local homeomorphism from $E$ into $X$. A section $\sigma$ is a function from an open set $U$ of $X$ to $E$ such that $p \circ \sigma =Id_U$. We say that the section is global if $U=X$. Sections are determined by their images, as $p$ is their common continuous inverse function. Besides, images of sections form a basis for the topology of $E$. We will refer indistinctly to the image set of a section and the function itself. 

  In what follows we assume that a metric language $\mathcal{L}$ is given and we omit the prefix $\mathcal{L}$ when talking about $\mathcal{L}$-formulas, $\mathcal{L}$-conditions and others. 

\begin{definition}[Sheaf of metric structures]\label{metricsheafdef}
Let $\X$ be a topological space.  A sheaf of metric structures $\A$ on
$\X$ consists of:
\begin{enumerate}
\item A sheaf $\: (E, p)$ over $X$.
\item For all $x$ in $X$ we associate a metric structure\\
$(\A_x, d) = \left( E_x, \{R_i^{(n_i)}\}_x , \{f_j^{(m_j)}\}_x, \{c_k\}_x, \Delta _{R_{i,x}}, \Delta _{f_{i,x}}, d, [0, 1]\right)$,\\
 where $E_x$ is the fiber $p^{-1}(x)$ over $x$, and the following conditions hold:
\begin{enumerate}
\item $(E_x, d_x)$ is a complete, bounded metric space of diameter $1$.
\item For all $i$, $R_i^\A =\bigcup_{x \in X} R_i^{\A_x}$ is a
  continuous function according to the topology of $\bigcup_{x \in X}
  E_x^{n_j}$. 
\item For all $j$, the function $f_j^\A = \bigcup_x f_j^{\A_x}
  :\bigcup_x E_x^{m_j} \to \bigcup_x E_x$ is a continuous function
  according to the topology of $\bigcup_{x \in X} E_x^{m_j}$.
\item For all $k$, the function $c_k^\A : X \to E$, given by $c_k^\A
  (x)=c_k^{\A_x}$, is a continuous global section.
\item We define the premetric function $d^\A$ by $d^\A=\bigcup _{x \in
    X}d_x: \bigcup _{x \in X} E_x^2 \to [0, 1]$, where $d^\A$ is a
  continuous function according to the topology of $\bigcup_{x \in X} E_x^2$.
\item For all $i$, $\Delta^\A_{ R_i}=\inf_{x \in X} (\Delta^{\A_x}_{R_i})$ with the condition that $\inf_{x \in X} \Delta^{\A_x}_{R_i}(\varepsilon)>0$ for all $\varepsilon>0$. 
\item For all $j$, $\Delta^\A_{ f_j}=\inf_{x \in X} (\Delta^{\A_x}_{f_i})$ with the condition that  $\inf_{x \in X} \Delta^{\A_x}_{f_i}(\varepsilon)>0$ for all $\varepsilon>0$.
\item The set $[0,1]$ is a second sort and is provided with the usual metric.
\end{enumerate}
\end{enumerate}
\end{definition}

 The space $\bigcup_x E_x^{n}$ has as open sets the image of sections given by $\bra \sigma_1, \dots, \sigma_n \ket= \sigma_1 \times \dots \times \sigma_n \cap \bigcup_x E_x^{n}$. These are the sections of a sheaf over $X$ with local homeomorphism $p^*$ defined by $p^*\bra \sigma_1(x), \dots, \sigma_n(x) \ket= x$. We drop the symbol $^*$ from our notation when talking about this local homeomorphism but it must be clear that this local homeomorphism differs from the function $p$ used in the definition of the topological sheaf.

The induced function $d^\A$ is not necessarily a metric nor a
pseudometric. Thus, we cannot expect the sheaf just defined to be a metric structure, in the sense of continuous logic. Indeed, we want to build the local semantics on the sheaf so that for a given sentence $\phi$, if $\phi$ is true at some $x \in X$, then we can find a neighborhood $U$ of
$x$ such that for every $y$ in $U$, $\phi$ is also true. In order to
accomplish this task, first note that semantics in
continuous logic is not defined on formulas but on conditions. Since
the truth of the condition ``$\phi<\varepsilon$'' for $\varepsilon$
small can be thought as a good approximation to the notion of $\phi$
being true in a first order model, one may choose this as the
condition to be forced in our metric sheaf.  Therefore, for a given
real number $\varepsilon \in (0,1)$, we consider conditions of the
form $\phi < \varepsilon$ and $\phi > \varepsilon$. Our first result
comes from investigating to what extent the truth in a fiber
``spreads'' onto the sheaf.
 
\begin{lemma}[Truth continuity in restricted cases]\label{conti}
\begin{itemize}
\item Let $\varepsilon$ be a real number, $x \in X$, $\phi$ an $\mathcal{L}$-formula composed only of the logical metric connectives and the quantifier $\inf$.  If $\A_x \models \phi (\sigma(x)) < \varepsilon$, then there exists an open neighborhood $U$ of $x$, such that for every $y$ in $U$, $\A_y \models \phi(\sigma(y)) < \varepsilon$.
\item Let $\varepsilon$ be a real number, $x \in X$, $\phi$ an $\mathcal{L}$-formula composed only of the logical metric connectives and the quantifier $\sup$. If $\A_x \models \phi (\sigma(x)) > \varepsilon$, then there exists an open neighborhood $U$ of $x$, such that for every $y$ in $U$, $\A_y \models \phi(\sigma(y)) > \varepsilon$.
\end{itemize}
\end{lemma}
\begin{proof}
 In atomic cases, use the fact that $d^\A$ and $R^\A$ are continuous with respect to the topology defined by sections.  For logical connectives this is a simple consequence of the fact that every connective is a continuous function. Thus, a formula $\phi(x_1, ..., x_2)$ constructed inductively only from connectives and atomic formulas is a composition of continuous functions and therefore continuous. If $\A_x \models \phi(\sigma_1(x), \dots, \sigma_n(x))< \varepsilon$, then $p\left( \bra \sigma_1, 
\dots, \sigma_2 \ket \cap \phi ^{-1} [0, \varepsilon)\right)$ is an open set in $X$ satisfying the condition that for all $y$ in it $\A_y \models \phi(\sigma_1(y), \dots, \sigma_n(y))< \varepsilon$.
\end{proof}

  In particular, the above lemma is true for $\mathcal{F}$-restricted sentences. We may consider a different proof by induction using the density of the reals. This alternative approach provides a better setting to define the point-forcing relation on conditions.

\begin{definition}[Point Forcing]
Given a metric sheaf $\; \A$ defined on some topological space $X$, and a real number $\varepsilon \in (0,1)$ we define the relation $\Vdash _x$ on the set of conditions $\phi < \varepsilon$ and $\phi > \varepsilon$ ($\phi$ an $\mathcal{L}$-sentence) and for an element $x \in X$ by induction as follows 

Atomic formulas
\begin{itemize}
\item $\A \Vdash _x d(\sigma_1, \sigma_2)<\varepsilon
\iff d_x(\sigma_1(x), \sigma_2(x))<\varepsilon
$
\item  
$
\A \Vdash _x d(\sigma_1, \sigma_2)>\varepsilon \iff d_x(\sigma_1(x), \sigma_2(x))>\varepsilon $
\item
$
\A \Vdash _x R(\sigma_1, \dots, \sigma_n)< \varepsilon 
\iff R^{\A_x}(\sigma_1(x), \dots, \sigma_n(x)) < \varepsilon
$
\item
$
\A \Vdash _x R(\sigma_1, \dots, \sigma_n)> \varepsilon \iff R^{\A_x}(\sigma_1(x), \dots, \sigma_n(x)) > \varepsilon$
\end{itemize}

Logical connectives
\begin{itemize}
\item
$
\A \Vdash _x \max(\phi, \psi)<\varepsilon
\iff \A \Vdash _x \phi < \varepsilon $ and $ \A \Vdash _x \psi <\varepsilon
$
\item
$
\A \Vdash _x \max(\phi, \psi)>\varepsilon 
\iff \A \Vdash _x \phi > \varepsilon $ or $ \A \Vdash _x \psi > \varepsilon
$
\item
$
\A \Vdash _x \min(\phi, \psi) < \varepsilon
\iff \A \Vdash _x \phi < \varepsilon $ or $ \A \Vdash _x \psi < \varepsilon
$
\item
$
\A \Vdash _x \min(\phi, \psi) > \varepsilon
\iff \A \Vdash _x \phi > \varepsilon $ and $ \A \Vdash _x \psi > \varepsilon
$
\item
$
\A \Vdash _x 1 \ts \phi < \varepsilon
\iff \A \Vdash _y \phi > 1 \dot{-} \varepsilon 
$
\item
$
\A \Vdash _x 1 \ts \phi > \varepsilon
\iff \A \Vdash _y \phi < 1 \dot{-} \varepsilon 
$
\item $\A \Vdash _x \phi \ts \psi < \varepsilon \iff $One of the following holds:\\
i) $\A \Vdash_x \phi < \psi $\\
\hspace{2.5cm} ii) $\A \nVdash_x \phi < \psi$ and $\A \nVdash_x \phi > \psi$\\
\hspace{2.5cm} iii) $\A \Vdash_x \phi > \psi $ and $ \A \Vdash_x \phi < \psi + \varepsilon$
\item $\A \Vdash _x \phi \ts \psi > \varepsilon \iff \A \Vdash _x \phi > \psi + \varepsilon$
\end{itemize}

Quantifiers
\begin{itemize}
\item
$
\A \Vdash _x \inf _\sigma \phi(\sigma) < \varepsilon
\iff $There exists a section $\mu$ such that $ \A \Vdash _x \phi (\mu )< \varepsilon$.
\item
$
\A \Vdash _x \inf _\sigma \phi(\sigma) > \varepsilon
\iff $
There exists an open set $U \ni x$ and a real number $\delta_x>0$ such that for every $y \in U$ and every section $\mu$ defined on $y$, $ \A \Vdash _y \phi (\mu )> \varepsilon + \delta_x$
\item
$
\A \Vdash _x \sup _\sigma \phi(\sigma) <\epsilon 
\iff 
$There exists an open set $U \ni x$ and a real number $\delta_x$ 
such that for every $y \in U$ and every section $\mu$ defined on $y$
 $\A \Vdash_y \phi(\mu) < \varepsilon - \delta_x$.
\item
$
\A \Vdash _x \sup _\sigma \phi(\sigma) >\epsilon 
\iff $ There exists a section $\mu$ defined on $x$ such that $
 \A \Vdash_x \phi(\mu) > \varepsilon $
\end{itemize}
\end{definition}

The above definition and the previous lemma lead to the equivalence between $\A \Vdash _x \inf _\sigma (1 \ts \phi) > 1 \ts \varepsilon $ and $\A \Vdash _x \sup _\sigma \phi < \varepsilon$. In addition, we can state the truth continuity lemma for the forcing relation on sections as follows.

\begin{lemma}
  Let $\phi(\sigma)$ be an $\mathcal{F}-$restricted formula. Then 
  \begin{enumerate}
  \item $\A \Vdash _x \phi(\sigma) < \varepsilon$ iff there exists $U$ open neighborhood of $x$ in $X$ such that  $\A \Vdash _y \phi(\sigma) < \varepsilon$ for all $y \in U$.
\item $\A \Vdash _x \phi(\sigma) > \varepsilon$ iff there exists $U$ open neighborhood of $x$ in $X$ such that  $\A \Vdash _y \phi(\sigma) > \varepsilon$ for all $y \in U$.
  \end{enumerate}
 \end{lemma}

We can also define the point-forcing relation for non-strict inequalities by
\begin{itemize}
\item $\A \Vdash _x \phi \leq \varepsilon$ iff $\A \nVdash _x \phi > \varepsilon $ and
\item $\A \Vdash _x \phi \geq \varepsilon$ iff $\A \nVdash _x \phi < \varepsilon $,
\end{itemize}
for $\mathcal{F}-$restricted formulas. This definition allows us to show the following proposition.

\begin{proposition}
 Let $0<\varepsilon' < \varepsilon$ be real numbers. Then
 \begin{enumerate}
 \item If $\A \Vdash_x \phi(\sigma) \leq \varepsilon '$ then $\A \Vdash_x \phi(\sigma) < \varepsilon$.
\item If $\A \Vdash_x \phi(\sigma) \geq \varepsilon$ then $\A \Vdash_x \phi(\sigma) > \varepsilon '$.
 \end{enumerate}
\end{proposition}
\begin{proof}
  By induction on the complexity of formulas.
\end{proof}

 The fact that sections may have different domains brings additional difficulties to the problem of defining a metric function with the triangle inequality holding for an arbitrary triple. However, we do not need to consider the whole set of sections of a sheaf but only those whose domain is in a filter of open sets (as will be evident in the construction of the ``Metric Generic Model'' below). One may consider a construction of such a metric by defining the ultraproduct and the ultralimit for an ultrafilter of open sets. However, the ultralimit may not be unique since $E$ is not always a compact set in the topology defined by the set of sections. In fact, it would only be compact if every fiber were finite. Besides, it may not be the case that the ultraproduct is complete. Thus, we proceed in a different way by observing that a pseudometric can be defined for the set of sections with domain in a given filter.

\begin{lemma}
Let $\mathbb{F}$ be a filter of open sets. For all sections $\sigma$ and $\mu$ with domain in $\mathbb{F}$, let the family $\mathbb{F}_{\sigma \mu}=\{U\cap \dom(\sigma)\cap \dom(\mu) | U \in \F\}$. Notice that $\phi \notin \F_{\sigma \mu}$. Then the function
\[
\rho_{\mathbb{F}}(\sigma, \mu)=\inf_{U \in \F _{\sigma \mu}} \;\sup_{x \in U} d_x (\sigma(x), \mu(x))
\]
is a pseudometric in the set of sections $\sigma$ such that $\dom(\sigma ) \in \mathbb{F}$.
\end{lemma}
\begin{proof}
 We prove the triangle inequality. Let $\sigma_1$, $\sigma_2$ and $\sigma_3$ be sections with domains in $\mathbb{F}$, and let $V$ be the intersection of their domains. Then it is true that
\[
\sup_{x \in V} d_x(\sigma_1(x), \sigma_2(x)) \leq \sup_{x \in V} d_x(\sigma_1(x), \sigma_3(x))+\sup_{x \in V} d_x(\sigma_3(x), \sigma_2(x)),
\] 
and since $\sup_{ x \in A} f(x) \leq \sup_{ x \in B} f(x)$ whenever $A \subset B$,
 \begin{equation*}
\inf_{W \in \mathbb{F}_{\sigma_1 \sigma_2}} \sup_{x \in W} d_x(\sigma_1(x), \sigma_2(x)) \leq \inf_{W \in \mathbb{F}_{\sigma_1 \sigma_2}}\left( \sup_{x \in W} d_x(\sigma_1(x), \sigma_3(x))+\sup_{x \in W} d_x(\sigma_3(x), \sigma_2(x))\right).
\end{equation*}
Given $\varepsilon> 0$, there exist $W'$ and $W''$ such that
\begin{align*}
\sup_{x \in W'} d_x(\sigma_1(x), \sigma_3(x)) &<\inf_{W \in \mathbb{F}_{\sigma_1, \sigma_3}}\sup_{x \in W} d_x(\sigma_1(x), \sigma_3(x)) +\varepsilon/2 \\
 \sup_{x \in W''} d_x(\sigma_3(x), \sigma_2(x)) &<\inf_{W \in \mathbb{F}_{\sigma_2, \sigma_3}}\sup_{x \in W} d_x(\sigma_3(x), \sigma_2(x)) +\varepsilon/2 .
\end{align*}
Therefore,
\begin{multline*}
\sup_{x \in W'\cap W''} d_x(\sigma_1(x), \sigma_3(x))+\sup_{x \in W' \cap W''} d_x(\sigma_3(x), \sigma_2(x)) <\\ \inf_{W \in \mathbb{F}_{\sigma_1 \sigma_2}}\sup_{x \in W} d_x(\sigma_1(x), \sigma_3(x)) + \inf_{W \in \mathbb{F}_{\sigma_2 \sigma_3}}\sup_{x \in W} d_x(\sigma_3(x), \sigma_2(x))+\varepsilon.
\end{multline*}
Since $W' \cap W''$ is in $\mathbb{F}_{\sigma_1 \sigma_2}$ and $\varepsilon$ was chosen arbitrarily, the triangle inequality holds for $\rho_\F(\sigma, \mu)$.
\end{proof}

In the following, whenever we talk about a filter $\mathbb{F}$ in $X$ we will be considering a filter of open sets. For any pair of sections $\sigma$, $\mu$ with domains in a filter, we define $\sigma \sim_\mathbb{F} \mu$  if and only if $\rho_{\mathbb{F}}(\sigma, \mu)=0$. This is an equivalence relation, and the quotient space is therefore a metric space under $d_\mathbb{F}([\sigma], [\mu] )= \rho_\mathbb{F} (\sigma, \mu)$.  The quotient space provided with the metric $d_\F$ is the metric space associated with the filter $\mathbb{F}$. If $\F$ is principal and the topology of the base space X is given by a metric, then the associated metric space of that filter is complete. In fact completeness is a trivial consequence of the fact that sections are continuous and bounded in the case of a $\sigma$-complete filter (if $X$ is a metric space). However, principal filters are not interesting from the semantic point of view and $\sigma$-completeness might not hold for filters or even ultrafilters of open sets. The good news is that we can still guarantee completeness in certain kinds of ultrafilters.

\begin{theorem}\label{regspace}
  Let $\A$ be a sheaf of metric structures defined over a regular topological space $X$. Let $\F$ be an ultrafilter of regular open sets. Then, the induced metric structure in the quotient space $\A[\F]$ is complete under the induced metric.
\end{theorem}

In order to prove this theorem we need to state a few useful lemmas.

\begin{lemma}\label{nonempty}
  Let $A$ and $B$ be two regular open sets. If $A \setminus B \neq \emptyset$ then $\interior (A \setminus B) \neq \emptyset$.
\end{lemma}
\begin{proof} 
  If $x\in A \setminus B$ and  $int(A\setminus B)= \emptyset$, then $x \in \overline{B}$ and $A \subset \overline{B}$. Therefore $A \subset \interior(\overline{B})=B$ which is in contradiction to the initial hypothesis.
\end{proof}

\begin{lemma}\label{cauchylimit}
Let $\F$ be a filter and $\{ \sigma_n\}$ be a Cauchy sequence of sections according to the pseudometric $\rho_\F$ with all of them defined in an open set $U$ in $\F$. Then
\begin{enumerate}
\item There exists a limit function $\mu_\infty$ not necessarily continuous defined on $U$ such that $\lim_{n \to \infty} \rho_\F(\sigma_n, \mu_\infty)=0$.
\item If $X$ is a regular topological space and $\interior(\text{\rm im}(\mu_\infty))\neq \emptyset$, there exists an open set $V \subset U$, such that $\mu_\infty \upharpoonright V$ is continuous.
\end{enumerate}
\end{lemma}
\begin{proof}
  \begin{enumerate}
  \item This follows from the fact that $\{ \sigma_n(x)\}$ is a Cauchy sequence in the complete metric space $(E_x, d_x)$. Then let $\mu_\infty(x)= \lim_{n \to \infty} \sigma_n (x)$ for each $x \in U$.
\item Consider the set of points $e$ in $\mu_{\infty}$ such that there exists a section $\eta$ defined in some open neighborhood $U$ of $x$, with $\eta (x)=e$ and $\eta \subset \sigma_\infty$. Let $V$ be the projection set in $X$ of that set of points $e$. This is an open subset of $U$ and $\mu_\infty \upharpoonright V$ is a section. 
\end{enumerate}
\end{proof}

We can now prove Theorem \ref{regspace}. 

\begin{proof}
  Let $\{[\sigma^m] | m \in \omega\}$ be a Cauchy sequence in the associated metric space of an ultrafilter of regular open sets $\mathbb{F}$. If the limit exists, it is unique and the same for every subsequence. Thus, we define the subsequence $\{[\mu^k] | k \in \omega\}$ by making $[\mu^k]$ equal to $[\sigma^m]$ for the minimum $m$ such that for all $n\geq m$, $d_\mathbb{F}([\sigma^m],[\sigma^n]) < k^{-1}$. Since $d_\mathbb{F}([\mu^k],[\mu^{k+1}]) < k^{-1}$, for every pair $(k, k+1)$, there exists an open set $U_k \in \F$, such that 
\[
\sup _{x \in U_k}d_x(\mu^k(x), \mu^{k+1}(x)) < k^{-1}.
\]
Let $W_1= U_1$, $W_m= \cap_{i=1}^m U_k$ and define a function $\mu_\infty$ on $W_1$ as follows.
\begin{itemize}
\item If $x \in W_k \setminus W_{k+1}$ for some $k$, let $\mu_\infty(x)= \mu^k(x)$.
\item Otherwise, if $x \in W_k$ for all $k$, we can take $\mu_\infty(x)=\lim_{k \to \infty}\mu^k(x)$.
\end{itemize}

The function $\mu_{\infty}$ might not be a section but, based on the above construction, one can find  a suitable restriction $\sigma_\infty$ that is indeed a section but defined on a smaller domain. We show this by analyzing different cases.
\begin{enumerate}
\item \label{item1} If $W_1=W_k$ for all $k$, $\bigcap_k W_k =W_1$ then for all $x$ in $W_1$, $\sigma_\infty(x)=\lim_{k \to \infty}\mu^k(x)$. 
\begin{enumerate}
\item \label{case1} Suppose $\interior (\mu_\infty) =\emptyset$. Let $\tilde B_1 = W_1$. For every $x$ in $B_1$ choose a section $\eta_x$, such that $\eta_x(x)=\mu_\infty(x)$; by the continuity of $d^\A$ in the sheaf of structures, the set $\tilde B_k= p(\langle \eta_x, \mu^k \rangle \cap (d^\A)^{-1}[0, k^{-1}))$ for $k \geq 2$ is an open neighborhood of $x$. Consider $\bigcap_{k\in \omega} \tilde B_k$. It is clear that this set is not empty and that $\interior (\bigcap_{k\in \omega} \tilde B_k) = \emptyset$, as we assumed that $\interior (\mu_\infty)= \emptyset$. Since the base space is regular, there exists a local basis on $x$ consisting of regular open sets. We can define a family $\{B_k\}$ of open regular sets so that
  \begin{itemize}
  \item $B_1 := \tilde B_1$
  \item $B_k \subset \tilde B_k$
  \item $B_{k+1} \subset B_{k}$
  \item $x \in B_{k}$
  \end{itemize}
   Let $C_1:=B_1=W_1$. For all $k \geq 2$, define  $C_{k+1} \subset C_k \cap B_{k+1}$ with the condition that $C_{k+1}$ is a regular open set and let $V_k \subset C_{k}\setminus C_{k+1}$ be some regular open set such that 
\begin{equation*}
\overline{V}_k \;\cap\; \overline{C_k \setminus C_{k+1}}\setminus \interior (C_k \setminus C_{k+1}) = \emptyset ,
\end{equation*}
(if $C_{k+1} \subsetneq C_k$, this is possible by Lemma \ref{nonempty}; if $C_{k+1} =C_{k}$ let $V_k=\emptyset$) - i.e., the closure of $V_k$ does not contain any point in the boundary of $C_k \setminus C_{k+1}$ ( Use Lemma \ref{nonempty} and the fact that $X$ is regular). Then $\bigcap_{k\in \omega} C_k \supset \{x\}$.
 Now, if necessary, we renumber the family ${V_k}$ so that all the empty choices of $V_k$ are removed from this listing. Let $ \Gamma= \Gamma_{odd}:=\bigcup_{k=1}^{ \infty} V_{2k-1}$ and observe that this is an open regular set: 
\begin{align*}
 \overline \Gamma &= \overline{\bigcup_{k \in \omega} V_{2k-1}}=\bigcup_{k \in \omega} \overline{V_{2k-1}}\\
\interior \bigl( \overline{\Gamma}\bigr)&= \interior \bigl( \bigcup_{k \in \omega} \overline{V_{2k-1}}\bigr)= \bigcup_{k \in \omega}\interior( \overline{V_{2k-1}})=\Gamma .
\end{align*}
For the first equality observe that if $z \in \overline{\bigcup_{k \in \omega} V_{2k-1}}$ then $z \in \overline V_{2l-1}$ for only one $l$ since $\overline{V_n}\cap \overline{V_m}=\emptyset$ for $m\neq n$. In the second line, if $z \in  \interior \bigl( \bigcup_{k \in \omega} \overline{V_{2k-1}}\bigr)$, then every open set containing $z$ is a subset of $\bigcup_{k \in \omega} \overline{V_{2k-1}}$ and again since $\overline{V_n}\cap \overline{V_m}=\emptyset$ for $m\neq n$, there exists at least an open neighborhood that is a proper subset of a unique $V_{2l-1}$.

 If $\Gamma$ is an element of $\F$, then we can define the section $\sigma_{\infty}$ in the open regular set $\Gamma$ by $\sigma_{\infty}\upharpoonright V_{2k-1} := \mu^{2k-1}\upharpoonright V_{2k-1}$. This is a limit section of the original Cauchy sequence. 

  Now, consider the family $\mathbb{G}$ of all $\Gamma$s that can be defined as above for the same element $x$ in $W_1$ and for the same family $\{C_k\}$. $\mathbb{G}$ is partially ordered by inclusion. Consider a chain $\{\Gamma_i\}$ in $\mathbb{G}$. Observe that $\bigcup_i \Gamma_i$ is an upper bound for this and that $\bigcup_i\Gamma_i$ is regular, since
\begin{gather*}
\bigcup_i\Gamma_i \subset \bigcup_i \interior\Bigl(\overline{C_{2i-1}\setminus C_{2i}}\Bigr)\\
\overline{C_{2i-1}\setminus C_{2i}}\:\cap\: \overline{C_{2i+1}\setminus C_{2i+2}}=\emptyset
\end{gather*}
 Thus, by Zorn Lemma there is a maximal element $\Gamma_{max}$. This intersects every element in the ultrafilter and therefore is an element of the same. Otherwise  we would be able to construct $\tilde \Gamma_{max} \supsetneq \Gamma_{max}$. Suppose that there exists $A \in \F$ such that $\Gamma_{max}\cap A = \emptyset$, then we could repeat the above arguments taking an element $x'$ in $W_1':= A \cap W_1= A \cap C_1$, finding $\Gamma' \subset W_1'$ with $\Gamma' \cap \Gamma_{max}=\emptyset$. Take  $\tilde \Gamma_{max}=\Gamma' \cup \Gamma_{max}$.

\item \label{case2} If $\interior (\mu_\infty) \neq \emptyset$, then $U=p(\interior (\mu_{\infty}))$ is an open subset of $W_1$. Observe that $W_1\setminus U$ is an open set that contains all possible points of discontinuity of $\mu_{\infty}$. If some regular open set $V \subset U$ is an element of the ultrafilter, then $\mu_\infty \upharpoonright V$ is a limit section. If that is not the case, $V = \interior(X\setminus U)$ is in the ultrafilter and $V \cap W_1$ is an open regular set where $\mu_\infty$ is discontinuous at every point. Proceed as in case \ref{case1}.
\end{enumerate}
\item If there exists $N$ such that for all $m>N$ $W_N=W_m$, we rephrase the arguments used in \ref{case1}, this time defining $\sigma_\infty$ in a subset of $W_N$. Also, if $\bigcap_{k \in \omega} W_k$ is open and nonempty, we follow the same arguments as in \ref{case2}.
\item If for all $k$ $W_{k+1}\neq W_k$ and $\interior (\bigcap_{k \in \omega} W_k) \neq \emptyset$, let $W_1'= \interior (\bigcap_{k \in \omega} W_k)$ and use the same line of argument as in cases \ref{case1} and \ref{case2}.
\item If $\interior (\bigcap_{k \in \omega} W_k) = \emptyset$, for all $k$ such that  $W_k \setminus W_{k+1} \neq \emptyset$ define $\sigma_\infty$ on some open regular set $\overline{V}_k \subset \interior (W_k \setminus W_{k+1})$ so that $\sigma_{\infty}\upharpoonright V_k=\mu^k \upharpoonright V_k$. Then, $\sigma_\infty$ is defined in $\bigcup_{k \in \omega} V_k$, and repeat the line of argument used in case \ref{case1}.
\end{enumerate}
 Note that for all $\mu^k$
\begin{itemize}
\item If $\sigma_\infty(x)= \mu^{k+n}(x)$, then $d_x(\sigma_\infty(x), \mu^k(x)) < k^{-1}$ for $x$ in the common domain.
\item If $\sigma_\infty(x)= \lim_{n \in \omega}\mu^{n}(x)$, then there exists $N$ such that for $m>N$ $d_x(\sigma_\infty(x), \mu^m(x))< k^{-1}$ and taking $m > k$, by the triangle inequality $d_x(\sigma_\infty(x), \mu^k(x))< 2 k^{-1}$.
\end{itemize}
This shows that 
\begin{equation*}
   \sup_{x \in W_k \cap \bigcup V_n} d_x(\sigma_\infty (x),\mu^k(x)) < 2(k-1)^{-1} 
 \end{equation*}
and then $\sigma_\infty$ is a limit section. Finally, check that $p \circ \sigma^\infty = Id_{\dom(\sigma_\infty)}$.
\end{proof}

Before studying the semantics of the quotient space of a generic filter, we define the relation $\Vdash_U$ of local forcing in an open set $U$ for a sheaf of metric structures. The definition is intended to make the following statements about local and point forcing valid
\begin{align*}
  \A \Vdash_U \phi(\sigma) < \varepsilon &\iff \forall x \in U \; \A \Vdash_x \phi(\sigma)< \delta \text{ and } \\
  \A \Vdash_U \phi(\sigma) > \delta &\iff \forall x \in U \; \A \Vdash_x \phi(\sigma)> \varepsilon,
\end{align*}
for some $\delta < \varepsilon$. This is possible as a consequence of the truth continuity lemma.

\begin{definition}[Local forcing for Metric Structures]
  Let $\A$ be a Sheaf of metric structures defined in $X$, $\varepsilon$ a positive real number, $U$ an open set in $X$, and $\sigma_1,\dots, \sigma_n$ sections defined in $U$. If $\phi$ is an $\mathcal{F}$- restricted formula the relations $\A \Vdash_U \phi(\sigma) < \epsilon$ and $\A \Vdash_U \phi(\sigma) > \epsilon$ are defined by the following statements

Atomic formulas
\begin{itemize}
\item $\A \Vdash_U d(\sigma_1, \sigma_2) < \varepsilon \iff \sup_{x \in U} d_x(\sigma_1(x), \sigma_2(x)) < \varepsilon$
\item $\A \Vdash_U d(\sigma_1, \sigma_2) > \varepsilon \iff \inf_{x \in U} d_x(\sigma_1(x), \sigma_2(x))> \varepsilon$.
\item $\A \Vdash_U R(\sigma_1,\dots, \sigma_n) < \varepsilon \iff \sup_{x \in U} R^{\A_x}(\sigma_1(x),\dots, \sigma_n(x)) < \varepsilon$
\item $\A \Vdash_U R(\sigma_1,\dots, \sigma_n) > \varepsilon \iff \inf_{x \in U} R^{\A_x}(\sigma_1(x),\dots, \sigma_n(x))> \varepsilon$
\end{itemize}

Logical connectives
\begin{itemize}
\item $\A \Vdash_U \max(\phi, \psi) < \varepsilon \iff$ $\A \Vdash_V \phi < \varepsilon$ and $\A \Vdash_W \psi < \varepsilon$
\item $\A \Vdash_U \max(\phi, \psi) > \varepsilon \iff$ There exist open sets $V$ and $W$ such that $V \cup W = U$ and $ \A \Vdash_V \phi > \varepsilon$ and $\A \Vdash_W \psi > \varepsilon$
\item $\A \Vdash_U \min(\phi, \psi) < \varepsilon \iff $  There exist open sets $V$ and $W$ such that $V \cup W = U$ and  $\A \Vdash_V \phi < \varepsilon$ and $\A \Vdash_W \psi < \varepsilon$ 
\item $\A \Vdash_U \min(\phi, \psi) < \varepsilon \iff \A \Vdash_U \phi < \varepsilon$ and $\A \Vdash_U \psi < \varepsilon$
\item $\A \Vdash_U 1 \ts \psi < \varepsilon \iff \A \Vdash_U \psi> 1 \ts \varepsilon$
\item $\A \Vdash_U 1 \ts \psi > \varepsilon \iff \A \Vdash_U \psi< 1 \ts \varepsilon$
\item $\A \Vdash_U \phi \ts \psi < \varepsilon \iff$ One of the following holds\\
i) $\A \Vdash_U \phi < \psi $\\
\hspace{2.5cm} ii) $\A \nVdash_U \phi < \psi$ and $\A \nVdash_U \phi > \psi$\\
\hspace{2.5cm} iii) $\A \Vdash_U \phi > \psi $ and $ \A \Vdash_U \phi < \psi + \varepsilon$
\item $\A \Vdash_U \phi \ts \psi > \varepsilon \iff$ $ \A \Vdash_U \phi > \psi + \varepsilon$
\end{itemize}

Quantifiers
\begin{itemize}
\item $\A \Vdash_U \inf_\sigma \phi(\sigma) < \varepsilon \iff $ there exist an open covering $\{U_i\}$ of $U$ and a family of section $\mu_i$ each one defined in $U_i$ such that $\A \Vdash_{U_i} \phi(\mu_i) < \varepsilon$  for all $i$
\item $\A \Vdash_U \inf_\sigma \phi(\sigma) > \epsilon \iff $   there exist $\varepsilon'$ such that $0< \varepsilon< \varepsilon'$  and an open covering $\{U_i\}$ of $U$ such that for every section $\mu_i$ defined in $U_i$ $\A \Vdash_{U_i} \phi(\mu_i) > \varepsilon'$
\item $\A \Vdash_U \sup_\sigma \phi(\sigma) < \varepsilon \iff $ there exist $\varepsilon'$  such that $0< \varepsilon'< \varepsilon$  and an open covering $\{U_i\}$ of $U$ such that for every section $\mu_i$ defined in $U_i$ $\A \Vdash_{U_i} \phi(\mu_i) < \varepsilon'$
\item $\A \Vdash_U \sup_\sigma \phi(\sigma) > \varepsilon \iff $ there exist an open covering $\{U_i\}$ of $U$ and a family of section $\mu_i$ each one defined in $U_i$ such that $\A \Vdash_{U_i} \phi(\mu_i) > \varepsilon$  for all $i$
\end{itemize}
\end{definition}

Observe that the definition of local forcing leads to the equivalences
\begin{align*}
 \A \Vdash _U \inf _\sigma (1 \ts \phi(\sigma)) > 1 \ts \varepsilon  &\iff \A \Vdash _U \sup _\sigma \phi(\sigma) < \varepsilon,\\
\A \Vdash _U \inf _\sigma (\phi(\sigma)) < \varepsilon  &\iff \A \Vdash _U \sup _\sigma (1 \ts \phi(\sigma)) > 1 \ts \varepsilon.
\end{align*}

 The fact that we can obtain a similar statement to the
 Maximum Principle of~\cite{CAIC} is even more important.

\begin{theorem}[Maximum Principle for Metric structures]
If $\A \Vdash_U \inf_{\sigma} \phi(\sigma) < \varepsilon$ then there exists a section $\mu$ defined in an open set $W$ dense in $U$ such that $\A \Vdash_U \phi(\mu) < \varepsilon'$, for some $\varepsilon' < \varepsilon$.
\end{theorem}

\begin{proof}
  That $\A \Vdash_U \inf_{\sigma} \phi(\sigma) < \varepsilon$ is equivalent to the existence of an open covering $\{U_i\}$ and a family of sections $\{ \mu_i \}$ such that $\A \Vdash_{U_i} \phi(\mu_i) < \varepsilon'$ for some $\varepsilon' < \varepsilon$. The family of sections $\mathcal{S}=\{\mu | \;\dom(\mu) \subset U \text{ and } \A \Vdash_{\dom(mu)} \phi(\mu) < \varepsilon' \}$ is nonempty and is partially ordered by inclusion. Consider the maximal element $\mu*$ of a chain of sections in $\mathcal{S}$. Then $\dom(\mu*)$ is dense in $U$ and $\A \Vdash_{\dom(\mu *)} \phi(\mu *) < \varepsilon'$.
\end{proof}

\section[Metric Generic Model Theorem]{The Metric Generic Model and its theory}\label{secgenmodel}

 In certain cases, the quotient space of the metric sheaf can be the
 universe of a metric structure in the same language as each of the
 fibers. We examine in this section one such case (sheaves of metric
 structures over \emph{regular} topological spaces, and generic for
 filter of \emph{regular} open sets. We do not claim that this is the
 optimal situation - however, we provide a proof of a version of a
 Generic Model Theorem for these Metric Generic models.

\begin{definition}[Metric Generic Model]
Let $\A=(X,p,E)$ be a sheaf of metric structures defined on a regular
topological space $X$ and $\F$ an ultrafilter of regular open sets in
the topology of $X$. We define the Metric Generic Model $\A[\F]$ by
\begin{equation*}
\A [\F]=\{ [\sigma]/ _{\sim_\F} | \dom(\sigma) \in \F\},
\end{equation*}
provided with the metric $ d_\F $ defined above, and
with
\begin{itemize}
\item 
\begin{equation*}
f^{\A [\F]}([\sigma_1]/ _{\sim_\F}, \dots, [\sigma_n]/ _{\sim_\F})=[f^{\A}(\sigma_1, \dots, \sigma_n)]/_{\sim_\F}
\end{equation*}
with modulus of uniform continuity $\Delta_{f}^{\A[\F]}=\inf_{x \in X} \Delta_{f}^{\A_x}$.
\item
\begin{equation*}
 R^{\A [\F]}([\sigma_1]/ _{\sim_\F}, \dots, [\sigma_n]/_{\sim_\F})=\inf_{U \in \F_{\sigma_1 \dots \sigma_n}}\: \sup_{x \in U} R_x(\sigma_1(x), \dots, \sigma_n(x))
\end{equation*}
with modulus of uniform continuity $\Delta_{R}^{\A[\F]}=\inf_{x \in X} \Delta_{R}^{\A_x}$.
\item 
\begin{equation*}
c^{\A [\F]}=[c]/_{\sim_\F}
\end{equation*}
\end{itemize}
\end{definition}

 Observe that the properties of $d_\F$ and the fact that $R^{\A}$ is continuous ensure that the Metric Generic Model is well defined as a metric structure. Special attention should be paid to the uniform continuity of $R^{\A[\F]}$. We prove this next:
 \begin{proof}
   It is enough to show this for an unary relation. First, suppose $d_\F([\sigma], [\mu])< \inf_{x \in X}\Delta_R^{\A_x}(\varepsilon)$, then
\begin{align*}
     \inf_{U \in \F_{\sigma \mu}}\sup_{x \in U}d_x(\sigma(x), \mu(x))&< \inf_{x \in X}\Delta_R^{\A_x}(\varepsilon)\\
\intertext{which implies that there exists $V \in \F_{\sigma \mu}$ such that }
\sup_{x \in V}d_x(\sigma(x), \mu(x))&< \inf_{x \in X}\Delta_R^{\A_x}(\varepsilon),
\intertext{and by the uniform continuity of each $R^{\A_x}$}
\sup_{x\in V}|R(\sigma(x))-R(\mu(x))|&\leq \varepsilon.
\intertext{We now state that}
\left|\inf_{U \in \F_\sigma}\:\sup_{x \in U}R(\sigma(x))-\inf_{U \in \F_\mu}\:\sup_{x \in U}R(\sigma(x))\right|&\leq \sup_{x\in V}\left|R(\sigma(x))-R(\mu(x))\right|.\\
\intertext{First consider $R^{\A[\F]}([\sigma]/_{\sim \F})\geq R^{\A[\F]}([\mu]/_{\sim \F})$} 
|R^{\A[\F]}([\sigma]/_{\sim \F})-R^{\A[\F]}([\mu]/_{\sim \F})|&=R^{\A[\F]}([\sigma]/_{\sim \F})- R^{\A[\F]}([\mu]/_{\sim \F})\\
&\leq \sup_{x \in V}R(\sigma(x))-R^{\A[\F]}([\mu]/_{\sim \F})
\intertext{Now, for all $\delta>0$ there exists $W \in \F_\mu$ such that }
\sup_{x\in W}R(\mu(x))&<\inf_{U \in \F_\mu}\:\sup_{x \in U}R(\mu(x))+\delta
\intertext{and indeed the same is true for $V'=V\cap W \in \F_{\sigma \mu}$. Therefore}
|R^{\A[\F]}([\sigma]/_{\sim \F})-R^{\A[\F]}([\mu]/_{\sim \F})|&\leq \sup_{x \in V'}R(\sigma(x))-\sup_{x\in V'}R(\mu(x))+\delta,
\intertext{where we have substituted $V$ by $V'$ in the first term since $V' \subset V$, and we can apply the same arguments to
 it. Also, since $\delta$ is arbitrary}
|R^{\A[\F]}([\sigma]/_{\sim \F})-R^{\A[\F]}([\mu]/_{\sim \F})|&\leq \sup_{x \in V'}R(\sigma(x))-\sup_{x\in V'}R(\mu(x))\\
&\leq \sup_{x \in V'}\left(R(\sigma(x))-R(\mu(x))\right)\\
&\leq \left|\sup_{x \in V'}\left(R(\sigma(x))-R(\mu(x))\right)\right|\\
&\leq \sup_{x \in V'}\left| R(\sigma(x))-R(\mu(x))\right|\leq \varepsilon
\end{align*}
In the case of $R^{\A[\F]}([\sigma]/_{\sim \F})\leq R^{\A[\F]}([\mu]/_{\sim \F})$ similar arguments hold.
 \end{proof}
 
 At this point the reader may find that part of the ``generality'' of
 the so called generic model is lost. This is indeed true and it is a
 consequence of the additional conditions that we have imposed on the
 topology of the base space (regularity) and on the ultrafilter to
 obtain a Cauchy complete metric space. However, we still refer to
 this model as generic to stress its resemblance to Caicedo's notion of
 generic model\cite{CAIC}. 

 We can now present the Generic Model Theorem (GMT) for metric structures. This provides a nice way to describe the theory of the metric generic model by means of the forcing relation and topological properties of the sheaf of metric structures. 

\begin{theorem}[Metric Generic Model Theorem]\label{genmodtheometric}
Let $\F$ be an ultrafilter of regular open sets on a regular topological space $X$ and $\A$ a sheaf of metric structures on $X$. Then
\begin{enumerate}
\item
\begin{equation*}
  \A[\F] \models \phi([\sigma]/_{\sim \F})< \varepsilon \iff \exists U \in \F \text{ such that } \A \Vdash_U \phi(\sigma)< \varepsilon
\end{equation*}
\item
  \begin{equation*}
     \A[\F] \models \phi([\sigma]/_{\sim \F})> \varepsilon \iff \exists U \in \F\text{ such that }\A \Vdash_U \phi(\sigma)> \varepsilon
  \end{equation*}
\end{enumerate}
\end{theorem}

\begin{proof}
 Atomic formulas
\begin{itemize}
\item  $\A[\F] \models d_\F(\eqclass{1}, \eqclass{2})< \varepsilon$
  iff $\inf_{U \in \F_{\sigma_1 \sigma_2}}\sup_{x \in U}
  d_x(\sigma_1(x), \sigma_2(x))< \varepsilon$. This is equivalent to
  saying that there exists $U \in \F_{\sigma_1 \sigma_2}$ such that
  \[\sup_{x \in U}  d_x(\sigma_1(x), \sigma_2(x))<
  \varepsilon.\]
 Further, by definition, that it is equivalent to $\A
  \Vdash_{U} d(\sigma_1, \sigma_2)< \varepsilon$.
\item For $\A[\F] \models R(\eqclass{1}, \dots,\eqclass{n})$ the same arguments as before apply.
\item  $\A[\F] \models d_\F(\eqclass{1}, \eqclass{2})> \varepsilon$ iff $\inf_{U \in \F_{\sigma_1 \sigma_2}}\sup_{x \in U} d_x(\sigma_1(x), \sigma_2(x))> \varepsilon$.
\begin{itemize}
\item ($\Rightarrow$) Let $\inf_{U \in \F_{\sigma_1 \sigma_2}}\sup_{x \in U} d_x(\sigma_1(x), \sigma_2(x))=r$ and $\varepsilon'=(r+\varepsilon)/2$. Then, the set $V=p(\bra\sigma_1,\sigma_2 \ket \cap d^{\A \; -1}(\varepsilon', 0])$ is nonempty and intersects every open set in $\F$. If $V \notin \F$, consider $X \setminus \overline{V}$. That set is not an element of $\F$ since $\dom(\sigma_1)\cap \dom(\sigma_2)\cap X \setminus \overline{V}$ would also be in $\F$  with $d_x(\sigma_1(x), \sigma_2(x)) \leq \varepsilon'$ for all $x$ in it. Therefore $\interior (\overline{V}) \cap \dom(\sigma_1)\cap \dom(\sigma_2)\in \F$ and for every element in this set $d_x(\sigma_1(x), \sigma_2(x)) \geq \varepsilon'$, which implies that there exists $U' \in \F$ such that $\inf_{x \in U'} d_x(\sigma_1(x), \sigma_2(x))\geq \varepsilon' > \varepsilon$.
\item($\Leftarrow$)  If $\A \Vdash_V d(\sigma_1, \sigma_2)>
  \varepsilon$ for some $V \in \F_{\sigma_1 \sigma_2}$,  then $V$
  intersects any open set in the generic filter and for any element in
  $V$, $d_x(\sigma_1(x), \sigma_2(x))\geq r$  where  $r=\inf_{x \in
    V}d_x (\sigma_1(x), \sigma_2(x))$. Thus, for all $U \in \F$,
 \[\sup_{x \in U}d_x(\sigma_1(x), \sigma_2(x))\geq r\] and therefore 
\[
\inf_{U\in \F_{\sigma_1 \sigma_2}}\sup_{x \in U}d_x(\sigma_1(x), \sigma_2(x))\geq r > \varepsilon.
\]
\end{itemize}
\item  Similar statements to those claimed above show the case of $\A[\F] \models R(\eqclass{1},\dots, \eqclass{n})> \varepsilon$.
\end{itemize}
Logical connectives
\begin{itemize}
\item For the connectives $1 \ts$, $\min$ and $\max$, it follows by simple induction in each case. We only show the proof for one of these connectives.
\begin{align*}
\A[\F]\models& \min \left(\phi(\eqclass{1}), \psi(\eqclass{2})\right)<\varepsilon \\
\iff &\A[\F]\models \phi(\eqclass{1})<\varepsilon \text{ or } \A[\F]\models \psi(\eqclass{2})\\
\substack{\text{ind} \\ \iff} & \exists U_1 \in \F \; \A \Vdash_{U_1} \phi(\sigma_1) < \varepsilon \text { or } \exists U_2 \in \F \; \A \Vdash_{U_2} \psi(\sigma_2)< \varepsilon\\
\iff & \exists U \in \F \text{ such that }\A \Vdash_{U} \min(\phi(\sigma_1),\psi(\sigma_2))< \varepsilon.
\end{align*}
\item  If $ \A[\F] \models \phi(\eqclass{1}) \ts \psi(\eqclass{2})< \varepsilon $ we analyze this by cases
\begin{itemize}
\item if $\A[\F] \models \phi(\eqclass{1}) < \psi(\eqclass{2})$
  \begin{align*}
   \iff & \exists r \text{ such that } \A[\F] \models \phi(\eqclass{1}) < r  \text { and } \A[\F] \models \psi(\eqclass{2})> r\\
\substack{\text{ind} \\ \iff} & \exists r \text{ such that } \exists U_1 \;\A \Vdash_{U_1} \phi(\sigma_1) < r  \text { and } \exists U_2 \; \A \Vdash_{U_2} \psi(\sigma_2)> r\\
\iff & \exists U \;\A \Vdash_U \phi(\sigma_1) <  \psi(\sigma_2)
  \end{align*}
\item If $\A[\F] \nvDash \phi(\eqclass{1}) < \psi(\eqclass{2})$ and $\A[\F] \nvDash \phi(\eqclass{1}) > \psi(\eqclass{2})$
  \begin{align*}
    \iff & \forall U \;\A \nVdash_U \phi(\sigma_1) <  \psi(\sigma_2) \text{ and } \forall U \;\A \nVdash_U \phi(\sigma_1) >  \psi(\sigma_2)
  \end{align*}
\item $\A \models  \phi(\eqclass{1}) > \psi(\eqclass{2})$ and $\A \models \phi(\eqclass{1}) < \psi(\eqclass{2})+\varepsilon$.
  \begin{multline*}
    \iff \exists U_1 \in \F \; \A \Vdash_{U_1}  \phi(\sigma_1) > \psi(\sigma_2) \text{ and }\\ \exists U_2 \in \F\; \A \Vdash_{U_2} \phi(\eqclass{1}) < \psi(\eqclass{2})+\varepsilon.
  \end{multline*}
\end{itemize}
\end{itemize}
Quantifiers
\begin{itemize}
\item
\begin{align*}
\A[\F] \models & \inf_{\eqclass{i}} \phi (\eqclass{i})< \varepsilon\\
&\iff  \exists \eqclass{1} \text{ such that }\A[\F]\models \phi(\eqclass{1})< \varepsilon\\
& \iff \exists U_1 \in \F \; \exists \sigma_1 \text{ such that }\A \Vdash_{U_1} \phi(\sigma_1)< \varepsilon\\
& \Rightarrow \exists U \in \F \text{ such that }\A \Vdash_{U} \inf_\sigma \phi(\sigma)< \varepsilon
\end{align*} 
For the other direction suppose that there exists $U \in \F$ such that $\A \Vdash_U \inf_\sigma \phi(\sigma)< \varepsilon$. Then the family $\mathcal{I}_\varepsilon= \{U \in \F | \A \Vdash_U \inf_\sigma \phi(\sigma)< \varepsilon \}$ is nonempty and can be partially ordered by the binary relation $\prec$ defined by: $U \prec V$ if and only if $U \supset V$. Consider the maximal element $U'$ of a chain defined in $\mathcal{I}_\varepsilon$.  Then there exists a covering $\{V_i\}$ of $U'$ all whose elements are basic open sets of class $\mathscr{C}$, and a family of sections $\{\mu_i\}$, such that $\A \Vdash_{V_i} \phi(\mu_i)< \varepsilon$. If any $V_i \in \F$ then $V_i=U'$. Otherwise it will contradict the maximality of $U'$. Also, if $\interior (X\setminus V_i) \in \F$ then $\A \Vdash_{\interior (X\setminus V_i) \cap U'} \inf_\sigma \phi(\sigma)< \varepsilon$ in contradiction to the maximality of $U'$. We conclude that there exists $\mu$ such that $\A \Vdash_{U'} \phi(\mu)< \varepsilon$.
\item  
\begin{align*}
\A[\F] \models & \sup_{\eqclass{i}} \phi (\eqclass{i})< \varepsilon\\
&\iff  \forall \eqclass{i} \; \A[\F]\models \phi(\eqclass{i})< \varepsilon\\
& \iff \forall \sigma_i \exists U_i \in \F \; \text{ such that }\A \Vdash_{U_i} \phi(\sigma_i)< \varepsilon
\end{align*} 
($\Rightarrow$) We prove this by contradiction. Suppose there exists $V$ in $\F$ such that for some $\sigma_i$ $\A \Vdash_V  \phi(\sigma_i) \geq \varepsilon $, then $V \cap U_i$ is also in $\F$ and in this set $\phi(\sigma_i) < \varepsilon$ and $\phi(\sigma_i) \geq \varepsilon$ are forced simultaneously.

($\Leftarrow$) Suppose that there exists $U \in \F$ such that $\A \Vdash_U \sup_\sigma \phi(\sigma)< \varepsilon$. Then the family $\mathcal{S}_\varepsilon= \{U \in \F | \A \Vdash_U \sup_\sigma \phi(\sigma)< \varepsilon \}$ is nonempty. The proof follows by similar arguments to those used in the case of $\inf_{\eqclass{i}} \phi (\eqclass{i})< \varepsilon$ above.
\end{itemize}
\end{proof}

We now stress that the Metric Generic Model Theorem (GMT) has distinct but strong connections with the Classical Theorem (see \cite{CAIC, FORE}).  In the case of the Metric GMT, we can observe similarities in the forcing definitions if we consider the parallelism between the minimum function and the disjunction, the maximum function  and the conjunction, the infimum and the existential quantifier. On the other hand, differences are evident if we compare the supremum with the universal quantifier. The reason for this is that in this case the sentence $1 \ts (1 \ts \phi)$, which is our analog for the double negation in continuous logic, is equivalent to the sentence $\phi$. Note that the point and local forcing definitions are consistent with this fact - i.e.,
\begin{align*}
\A \Vdash_U 1 \ts (1 \ts \phi)< \varepsilon \iff \A \Vdash_U \phi <  \varepsilon ,\\
\A \Vdash_U 1 \ts (1 \ts \phi)> \varepsilon \iff \A \Vdash_U \phi >  \varepsilon .
\end{align*}
As another consequence, the metric version of the GMT does not require an analog definition to the G\" odel translation. 

 We close this section by introducing a simple example that illustrates some of the elements just described. We study the metric sheaf for the continuous cyclic flow in a torus.

 Let $X= S^1$, $E= S^1 \times S^1$ and $p = \pi _1$, be the projection function onto the first component. Then, we have $E_q= S^1$. Given a set of local coordinates $x_i$ in $S_i$ and a smooth vector field $V$ on $E$, such that
 \begin{align*}
   V= V_1\frac{\partial }{\partial x_1}+V_2\frac{\partial}{\partial x_2}&&V_1(p) \neq 0 \;\; \forall p \in S^1,
 \end{align*}
 we can take as the set of sections the family of integrable curves of $V$. The open sets of the sheaf can be described as local streams through $E$. Complex multiplication in every fiber is continuously extended to a function between integral curves. Every section can be extended to a global section.

  Let us study the metric generic model of this sheaf. Note that $X$ is a topological regular space and that it admits an ultrafilter $\F$ of regular open sets. First, observe that $\A[\F]$ is a proper subset of the set of local integrable curves. In fact, every element in $\A [\F]$ can be described as the equivalence class of a global section in $E$:
For any element $[\sigma]\in  \A[\F]$, $U = dom (\sigma) \in \F$, and there exists a global integral curve $\mu$ in $E$ such that $\rho_\F(\sigma, \mu)=0$. This result leads to the conclusion that every ultrafilter filter of open sets in $S^1$ generates the same universe for $\A[\F]$. 
  Observe that every fiber can be made a metric structure with a metric given by the length of the shortest path joining two points. This, of course, is a Cauchy complete and bounded metric space. Dividing the distance function by $\pi$, we may redefine this to make $d(x,y)\leq 1$, for $x$ and $y$ in $S^1$. Therefore, this manifold is also a metric sheaf. In addition, observe that complex multiplication in $S^1$ extends to the sheaf as a uniformly continuous function in the set of sections.
  For any element $[\sigma] \in \A [\F]$, let $U= dom (\sigma )\in \F$ and $\mu$ be the global integral curve that extends $\sigma$. Thus, for arbitrary $\varepsilon >0$
\begin{align*}
  \A \Vdash _U d^\A(\sigma, \mu) < \varepsilon& &\text{ and as a consequence}& & \A[\F] \models d^{\A[\F]}([\sigma], [\mu]) = 0. 
\end{align*}
 In addition, the metric generic model satisfies the condition that the {\sl multiplication} between sections be left continuous. Let $\eta$ and $\mu$ be sections whose domain is an element of the ultrafilter. For any $\varepsilon < 1/2$, if 
\[
\A \Vdash _{\dom(\eta)\cap \dom(\mu)} d(\eta, \mu) < \varepsilon
\]
then for any other section $\sigma$ defined in an element of $\F$, it is true that in $V=\dom(\eta)\cap \dom(\mu)\cap \dom(\sigma)$ 
\begin{gather*}
\A \Vdash _V  d(\eta \sigma, \mu \sigma) < \varepsilon
\intertext{and also}
\A \Vdash _V  1 \ts \max (d(\eta, \mu), 1 \ts d(\eta \sigma, \mu \sigma)) < \varepsilon.
 \end{gather*}
By the metric GMT, we can conclude that
\[
\A[\F]\models 1 \ts \max (d^{\A[\F]}([\eta], [\mu]), 1 \ts d([\eta] [\sigma], [\mu] [\sigma])) < \varepsilon
\]
and since $\sigma, \eta$ and $\mu$ were chosen arbitrarily.
\[
\A[\F]\models \sup_\sigma \sup_\eta \sup_\mu \bigl[ 1 \ts \max (d^{\A[\F]}([\eta], [\mu]), 1 \ts d([\eta] [\sigma], [\mu] [\sigma]))\bigr] < \varepsilon .
\]
Right continuity, left invariance and right invariance of this metric can be expressed in the same fashion.

\section{Further directions}

Our theorem and constructions open up a new line of research with many
questions, mainly of two kinds: applications and model theory.

\subsection{Applications}

Besides our simple example, we expect there to be applications to
further \textbf{dynamical systems} obtained as sheaves of metric
structures, where the dynamics is provided by the behavior of (some
carefully chosen)
sections. More specifically, we expect that our results will be useful
for various constructions of sheaves over topological spaces naturally
associated to actions of compact groups over certain varieties - among
other variations. Additionally, the first author is currently working
in setting up applications to classical mechanics and possibly
quantization. Other possible applications include Zilber's Structural
Approximation (see~\cite{ZILB}). 

\subsection{Model Theory}

The Model Theoretic analysis of the new objects constructed here
(the metric generic model in section~\ref{secgenmodel} and the sheaf
of metric structures itself) has so far been analyzed from a model
theoretic perspective only up to the consequences of the Generic Model
Theorem (this is also true of the work done in the case of sheaves
with discrete First Order fibers). In particular, the stability theory
of sheaves, the construction of ``elementary'' extensions of these, of
sheaf theoretical versions of typespaces (typesheaves?) is quite
undeveloped so far (exceptions include the work of Montoya~\cite{MONT}
in some specific cases).

Additionally, extensions of this work to topological structures
(fibers with topological spaces, built in the framework of Flum and
Ziegler~\cite{FLUM}) have been explored and will appear in forthcoming
work of the two authors. Further lines of research (sheaves with
fibers that are more general than metric or different from topologic
in the style of Flum and Ziegler - for instance, fibers that are
measure algebras,
etc.) are yet to be explored. Finally, the connections to the Model
Theory of Metric Abstract Elementary Classes (see~\cite{HIRV}
and~\cite{ZAMB}) and in particular to topological dynamical
perspectives on typespaces, will provide relevant directions when
combined with our metric sheaves.

\bibliographystyle{jloganal}
\bibliography{sheaf}
\end{document}